\documentclass[a4,11pt]{article}
\usepackage{amsthm}

\usepackage[a4paper]{geometry}
\newcommand{\newnumbered}[2]{\newtheorem{#1}[theorem]{#2}}
\newcommand{\newunnumbered}[2]{\newtheorem{#1}[theorem]{#2}}
\newcommand{\Title}[2]{\title{#1}\newcommand{\Acknowledgements}{\section*{Acknowledgements} #2}}
\newcommand{\Author}[2][]{\author{#2}}
\newcommand{\Comma}{\and}
\newcommand{\Und}{\and}
\newcommand{\br}{, }
\newcommand{\fs}{. }
\newcommand{\thanksone}[3][]{#1\thanks{#3\email{\tt #2}}}
\newcommand{\thankstwo}[3][]{#1\thanks{#3\email{\tt #2}}}
\newcommand{\thanksthree}[3][]{#1\thanks{#3\email{\tt #2}}}

\newcommand{\email}[1]{#1}
\newcommand{\classno}[2][2000]{}

\newcommand{\mktitle}{\maketitle}
\newcommand{\mkabstitle}{}

\usepackage[utf8x]{inputenc}
\usepackage{amsfonts}
\usepackage{amsmath}
\usepackage{amssymb}
\usepackage[mathscr]{eucal}
\usepackage{color}
\usepackage[colorlinks=true,citecolor=black,linkcolor=black,urlcolor=blue]{hyperref}
\usepackage{comment}
\newtheorem{theorem}{Theorem}[section] 
\newtheorem{maintheorem}{Theorem}
\newtheorem{lemma}[theorem]{Lemma}     
\newtheorem{corollary}[theorem]{Corollary}
\newtheorem{proposition}[theorem]{Proposition}

\newnumbered{assertion}{Assertion}    
\newnumbered{conjecture}{Conjecture}  
\newnumbered{definition}{Definition}
\newnumbered{hypothesis}{Hypothesis}
\newnumbered{remark}{Remark}
\newnumbered{note}{Note}
\newnumbered{observation}{Observation}
\newnumbered{problem}{Problem}
\newnumbered{question}{Question}
\newnumbered{algorithm}{Algorithm}
\newnumbered{example}{Example}
\newnumbered{claim}{Claim}
\newunnumbered{notation}{Notation} 
\numberwithin{equation}{section}

\newcommand{\noop}[1]{}

\newcommand{\fiber}[1]{\mathsf{F}(#1)}
\newcommand{\rel}[1]{\mathsf{S}(#1)}
\newcommand{\kk}[1]{\mathcal{#1}}
\newcommand{\wl}[1]{\mathsf{WL}(#1)}
\newcommand{\qo}[1]{#1/e}

\newcommand{\eqv}{e_{\kk{X}}}
\newcommand{\df}[1]{\textbf{#1}}

\def\thrm{\begin{theorem}}
\def\thrml#1{\begin{theorem}\label{#1}}
\def\ethrm{\end{theorem}}
\def\rmrk{\begin{remark}}
\def\rmrkl#1{\begin{remark}\label{#1}}
\def\ermrk{\end{remark}}
\def\dfntn{\begin{definition}}
\def\dfntnl#1{\begin{definition}\label{#1}}
\def\edfntn{\end{definition}}
\def\nmrt{\begin{enumerate}}
\def\enmrt{\end{enumerate}}

\def\qtn{\begin{equation}}
\def\qtnl#1{\begin{equation}\label{#1}}
\def\eqtn{\end{equation}}
\def\lmm{\begin{lemma}}
\def\lmml#1{\begin{lemma}\label{#1}}
\def\elmm{\end{lemma}}
\def\crllr{\begin{corollary}}
\def\crllrl#1{\begin{corollary}\label{#1}}
\def\ecrllr{\end{corollary}}
\def\css{\begin{cases}}
\def\ecss{\end{cases}}

\def\cX{{\cal X}}
\def\cY{{\cal Y}}

\def\fK{{\frak K}}

\def\sF{\mathsf{F}}
\def\sS{\mathsf{S}}
\def\sT{\mathsf{T}}

\def\bR{\boldsymbol{r}}

\DeclareMathOperator{\aut}{Aut}

\DeclareMathOperator{\Aiso}{Iso_{alg}}

\DeclareMathOperator{\sym}{Sym}
\DeclareMathOperator{\wL}{\mathsf{WL}}

\def\ov{\overline}

\allowdisplaybreaks

\begin{document}

\Title{The Weisfeiler-Leman dimension of distance-hereditary graphs}{%
Alexander Gavrilyuk is supported by
Basic Science Research Program through the National Research Foundation of Korea (NRF) funded
by the Ministry of Education (grant number NRF-2018R1D1A1B07047427).
Roman Nedela is supported by the Czech Science Foundation, grant GACR 20-15576S
and by the Slovak Research and Development Agency, Grant No. APVV-15-0220.
}

\Author[Alexander L. Gavrilyuk, Roman Nedela, Ilia Ponomarenko]{%
\thanksone[Alexander L. Gavrilyuk]{alexander.gavriliouk@gmail.com}{%
Pusan National University\br
Busan, Republic of Korea\fs
}
\Comma
\thankstwo[Roman Nedela]{nedela@ntis.zcu.cz}{%
Department of Mathematics\br
University of West Bohemia\br
Pilsen, Czech Republic\fs
}
\Und
\thanksthree[Ilia Ponomarenko]{inp@pdmi.ras.ru}{%
St. Petersburg Department of the Steklov Mathematical Institute\br
St. Petersburg, Russia;
School of Mathematicsand Statistics of Central China Normal University\br
Wuhan, China\fs
}}

\classno{05E30 (primary), 05B15 (secondary)}

\date{\today}

\mktitle

\begin{abstract}
A graph is said to be distance-hereditary if the distance function in every
connected induced subgraph is the same as in the graph itself. We prove
that the ordinary Weisfeiler-Leman algorithm correctly tests the isomorphism
of any two graphs if one of them is distance-hereditary; more precisely,
the Weisfeiler-Leman dimension of the class of finite distance-hereditary
graphs is equal to~$2$.  The previously best known upper bound for the dimension was~$7$.
\end{abstract}

\mkabstitle

\section{Introduction}
Over the past few decades, the Weisfeiler-Lehman algorithm (WL) has become
one of the most studied tools for testing isomorphism of finite graphs~\cite{Conf2018}.
This algorithm colors the arcs of the graphs in question and
then compares the numerical invariants of the obtained colorings; the graphs
are declared to be isomorphic if the corresponding invariants are equal,
and nonisomorphic otherwise. In the general case, the output  is not always true,
for example, if the input graphs are nonisomorphic strongly regular graphs with the same parameters.

Stronger isomorphism invariants are obtained if, instead of coloring the arcs,
one considers coloring the $d$-tuples of vertices, $d>2$;
the corresponding generalization is called the $d$-dimensional Weisfeiler-Lehman algorithm or the $d$-dim WL for short.
It was introduced by Babai and played an essential role in his recent
quasipolynomial algorithm testing isomorphism of arbitrary graphs~\cite{B2015}.
For $d=1$ and $d=2$, the $d$-dim WL coincides with the naive refinement
and ordinary WL, respectively.

It can be shown that, given a graph $X$, there exists a positive integer $d_X$ such that
if $d\geq d_X$, then the $d$-dim WL correctly identifies $X$ (i.e., tests isomorphism between $X$ and any other graph);
the smallest such $d_X$ is called the WL-dimension of the graph $X$.
(For the exact definitions, we refer the reader to Section \ref{sect:pre}.)
An equivalent definition of the WL-dimension can also be done
in terms of the first order logic with counting quantifiers and bounded number of variables;
the interested reader is referred to the monograph~\cite{Grohe2017}.

A rather general problem can be formulated as follows: determine the maximum WL-dimension
of a graph belonging to a given class~$\fK$; this number is called the WL-dimension
of $\fK$ (cf.~\cite[Definition~18.4.3]{Grohe2017}). Although the WL-dimension
of the class of all graphs cannot be bounded by a constant \cite{CFI1992},
for many natural graph classes the situation is different. Among these classes are
the interval graphs~\cite{EPT2000}, the planar graphs~\cite{KPS2017}, and many others
(see, e.g., \cite{GK2019} and references therein).

In a recent paper  \cite{GN2019}, it was proved that the WL-dimension of the class
of graphs of rank width at most~$r$ is less than or equal to~$3r+4$. From~\cite[Proposition~7.3]{Oum2005},
it follows that if  $r=1$, then the latter class coincides with the well-known class
of  distance-hereditary graphs introduced in \cite{H1977}; a graph is said to be
distance-hereditary if the distance function in every connected induced subgraph is
the same as in the graph itself.  Thus, according to \cite{GN2019}, the WL-dimension
of the class of distance-hereditary graphs is at most~$7$. The main result
of the present paper shows that this upper bound is not tight. More precisely,
the following theorem holds.

\begin{maintheorem}\label{theo:main}
  The WL-dimension of the class of finite distance-hereditary graphs is equal to~$2$.
\end{maintheorem}

Note that when the dimension of a graph class $\fK$ is bounded from above by a constant $d$,
the graph isomorphism problem restricted to $\fK$ is solved in polynomial time by the $d$-dim WL.
Thus, Theorem \ref{theo:main} shows that this conclusion holds with $d=2$ if $\fK$ is the class
of distance-hereditary graphs. An efficient algorithm for this particular graph isomorphism problem
was constructed in~\cite{NUU}; see also \cite{DER}.

%
%

Modulo a characterization of graphs that have WL-dimension~$1$ (see \cite{AKRV2017}),
the proof of Theorem~\ref{theo:main} reduces to verify that the WL-dimension of a distance-hereditary graph
is at most~$2$. To this aim, we use  theory of coherent configurations, see Section \ref{sect:pre}.
Namely, given a graph~$X$ the output coloring
of the ordinary WL defines a coherent configuration~$\cX$, which preserves all information
needed to test isomorphism between $X$ and any other graph. Moreover, the invariants
of the coloring form a full invariant of~$\cX$ with respect to algebraic isomorphisms.
As was proved in \cite{FKV2019}, the WL-dimension of the graph~$X$ is at most $2$
if and only if the coherent configuration~$\cX$ is separable, i.e., every algebraic isomorphism
of~$\cX$ is induced by a suitable combinatorial isomorphism. Thus, we only need to check that
$\cX$ is separable if $X$ is a distance-hereditary graph.
The proof of the latter is based on an inductive characterization
of the distance-hereditary graphs, see~\cite[Theorem~1]{BM1986}.

To make the paper self-contained, we introduce relevant concepts and statements
of the theory of coherent configurations in Section~\ref{sect:pre}.
A translation of graph theoretical operations (used in  the inductive characterization of
the distance-hereditary graphs) to the language of coherent configurations occupies Sections \ref{sect:twins} and \ref{sect:main}.
The proof of Theorem~\ref{theo:main} is given in Section \ref{sect:proof}.
\smallskip

\section{Rainbows, Coherent configurations, Graphs}\label{sect:pre}
In our presentation of coherent configuration, we mainly follow the monograph~\cite{CP2019},
where all the details can be found.

\subsection{Notation}
Throughout the paper, $\Omega$ denotes a finite set. For $\Delta\subseteq \Omega$, the diagonal of the Cartesian product $\Delta\times\Delta$ is denoted by~$1_\Delta$.

For a binary relation $r\subseteq\Omega\times\Omega$, we set $\Omega_{-}(r)=\{\alpha\in\Omega\!:\ \alpha r\ne\varnothing\}$,  $\Omega_{+}(r)=\Omega_{-}(r^*)$, $r^*=\{(\beta,\alpha)\!:\ (\alpha,\beta)\in r\}$, $\alpha r=\{\beta\in\Omega\!:\ (\alpha,\beta)\in r\}$ for all $\alpha\in\Omega$, and $r^f=\{(\alpha^f,\beta^f)\!:\ (\alpha,\beta)\in r\}$ for any bijection $f$ from $\Omega$ to another set. The product of the relations $r,s\subseteq\Omega\times\Omega$, is denoted by $r\cdot s=\{(\alpha,\beta)\!:\ (\alpha,\gamma)\in r,\ (\gamma,\beta)\in s$ for some $\gamma\in\Omega\}$.

For a set $S$ of relations on $\Omega$, we denote by $S^\cup$ the set of all unions of the elements of $S$,
put $S^*=\{s^*\!:\ s\in S\}$, and $S^f=\{s^f\!:  s\in S\}$ for any bijection $f$ from $\Omega$ to another set.
For  $r\subseteq\Omega\times\Omega$, we define $r\cdot S=\{r\cdot s\!: s\in S\}$,
$S\cdot r=\{s\cdot r\!: s\in S\}$, and $\alpha S=\cup_{s\in S}\alpha s$, $\alpha\in \Omega$.

For a class $\Delta$ of a partition $\pi$, we set $\pi\setminus \Delta=\pi\setminus\{\Delta\}$.

\subsection{Rainbows}
Let $\Omega$ be a finite set and  $\sS$ a partition of $\Omega\times\Omega$. A pair $\cX=(\Omega,\sS)$ is called
a \df{rainbow} on $\Omega$ if
\begin{equation}\label{eq-rainbow}
1_\Omega\in \sS^\cup, \textrm{~and~} \sS^*=\sS.
\end{equation}
The elements of the sets $\Omega$, $\sS=:\sS(\cX)$,
and $\sS^{\cup}$ are called the \df{points},  \df{basis relations},  and  \df{relations} of~$\cX$, respectively.
A unique basic relation containing a pair $(\alpha,\beta)\in\Omega\times\Omega$ is denoted by $\bR_\cX(\alpha,\beta)$;
we omit the subscript~$\cX$ wherever it does not lead to misunderstanding.

A set $\Delta\subseteq\Omega$ is called a  \df{fiber} of a rainbow~$\cX$ if $1_\Delta\in \sS$;
the set of all fibers is denoted by $\sF:=\sF(\cX)$. The point set~$\Omega$ is the disjoint union of fibers.
If $\Delta$ is a union of fibers, then the pair
\[
\cX_\Delta=(\Delta,\sS_\Delta)
\]
is a rainbow, where $\sS_\Delta$ consists of all $s_\Delta=s\cap(\Delta\times\Delta)$, $s\in\sS$.
In what follows, we set $\cX\setminus\Delta=\cX_{\Omega\setminus\Delta}$.

Let $\cX=(\Omega,\sS)$ and  $\cX'=(\Omega',\sS')$ be rainbows. A bijection  $f\!:\Omega\to\Omega'$
is called a \df{combinatorial isomorphism} (or simply isomorphism) from $\cX$ to $\cX'$ if $\sS^f=\sS'$.
When $\cX=\cX'$, the set of all these isomorphisms form a permutation group on~$\Omega$.
This group has a (normal) subgroup
$$
\aut(\cX)=\{f\in\sym(\Omega)\!:\ s^f=s\ \,\text{for all}\ \, s\in \sS\}
$$
called the  \df{automorphism group} of~$\cX$.

\subsection{Coherent configurations}
A rainbow $\cX=(\Omega,\sS)$ is called a  \df{coherent configuration} if, for any $r,s,t\in \sS$, the number
\[
c_{rs}^t:=|\alpha r\cap \beta s^*|
\]
does not depend on the choice of $(\alpha,\beta)\in t$;  the numbers~$c_{rs}^t$ are called
the  \df{intersection numbers} of~$\cX$. In this case, the set $\sS^\cup$ contains
the relation $r\cdot s$ for all $r,s\in \sS^\cup$;  this relation is obviously the union
(possibly empty) of those $t\in\sS$ for which $c_{rs}^t\ne 0$.

Let $\cX$ be a coherent configuration. Then for any $s\in \sS$, the sets $\Omega_{-}(s)$ and $\Omega_{+}(s)$
are the fibers of~$\cX$. In particular, the union
$$
\sS=\bigcup_{\Delta,\Gamma\in \sF(\cX)}\sS_{\Delta,\Gamma}
$$
is disjoint, where $\sS_{\Delta,\Gamma}$ consists of all $s\in \sS$, contained in $\Delta\times\Gamma$ .
The number $|\delta s|$ with $\delta\in\Delta$ equals the intersection number $c_{ss^*}^{1_\Delta}$,
and hence does not depend on the choice of the point~$\delta$. It is called the  \df{valency} of $s$ and
denoted by $n_s$. 


\subsection{Algebraic isomorphisms and separability}
Let $\cX=(\Omega,\sS)$ and $\cX'=(\Omega',\sS')$ be coherent configurations.
A bijection $\varphi\!:\sS\to \sS',\ r\mapsto r'$ is called an \df{algebraic isomorphism} from~$\cX$ onto~$\cX'$ if
\qtnl{f041103p1}
c_{rs}^t=c_{r's'}^{t'},\qquad r,s,t\in \sS;
\eqtn
the set of all such $\varphi$ is denoted by $\Aiso(\cX,\cX')$.

Each isomorphism~$f$ from~$\cX$ onto~$\cX'$ induces an algebraic isomorphism between these configurations,
which maps $r\in \sS$ to $r^f\in \sS'$. A coherent configuration~$\cX$ is said to be  \df{separable}
if  every algebraic isomorphism from $\cX$ to another coherent configuration  is induced by a suitable bijection
(which in this case is an isomorphism of the configurations in question).

The algebraic isomorphism $\varphi$ induces a bijection from $\sS^\cup$ onto $(\sS')^\cup$:
the union $r\cup s\cup\cdots$ of basis relations of $\cX$ is taken to $r'\cup s'\cup\cdots$.
This bijection is also denoted by $\varphi$. It preserves the dot product, i.e.,
$\varphi(r\cdot s)=\varphi(r)\cdot\varphi(s)$ for all $r,s\in\sS$.

One can see that if $\Delta\in\sF(\cX)$, then  $\varphi(1_{\Delta^{}})=1_{\Delta'}$ for some $\Delta'\in\sF(\cX')$ (and $|\Delta|=|\Delta'|)$;
we also denote such a $\Delta'$ by $\Delta^\varphi$. This extends $\varphi$ to a bijection $\sF(\cX)\to \sF(\cX')$
so that $(1_{\Delta})'=1_{\Delta'}$ for all~$\Delta$.

\subsection{Parabolics and quotients}\label{ss:parabolics}
Let $\cX=(\Omega,\sS)$ be a rainbow.
An equivalence relation that is an element of $\sS^\cup$ is called a \df{parabolic} of~$\cX$.
The parabolic $1_\Omega$ is said to be trivial.
Suppose further that $\cX$ is a coherent configuration.
An important property of a parabolic $e$ is that if $\varphi$ is an algebraic isomorphism
from $\cX$ to a coherent configuration~$\cX'$, then $\varphi(e)$ is a parabolic of~$\cX'$
and
\qtnl{220320a}
|\alpha e|=|\alpha' e'|,\qquad \textrm{~for all~}\alpha\in\Delta,\ \alpha'\in\Delta',\ \Delta\in\sF(\cX),
\eqtn
where $e'=\varphi(e)$ and $\Delta'=\Delta^\varphi$. When $\cX=\cX'$ and $\varphi$ is the identical mapping,
this shows that the  classes of the equivalence relation $e$ restricted to $\Delta\in\sF$ have the same cardinality.

Let $e$ be an equivalence relation on~$\Omega$. Denote by $\Omega/e$ the set of all classes of~$e$.
The map
\begin{equation}\label{eq-ro}
  \rho_e\!:\Omega\to\Omega/e,\  \alpha\mapsto\alpha e,
\end{equation}
is obviously a surjection. It induces a natural surjection, also denoted by~$\rho_e$,
from the binary relations on $\Omega$ to those on $\Omega/e$,
in particular, $\sS/e=\{\rho_e(s)\!:\ s\in\sS\}$.
Given a partition $\pi$ of $\Omega$, we set $\pi/e$ to be the partition of $\Omega/e$ with classes $\rho_e(\Delta)$, $\Delta\in\pi$.

Suppose that $e$ is a parabolic of~$\cX$.
Then the pair
$$
\cX/e=
(\Omega/e,\sS/e)
$$
is a coherent configuration. The mapping $\rho_e$ induces a surjection
from the parabolics (respectively, fibers) of $\cX$ on those of~$\cX/e$.  Every algebraic isomorphism $\varphi$ from $\cX$
onto a coherent configuration~$\cX'$ induces a natural algebraic isomorphism from $\cX/e$ onto~$\cX'/e'$, taking $\rho_{e^{}}(s)$ to  $\rho_{e'}(\varphi(s))$ for all $s\in\sS$, where $e'=\varphi(e)$.
Further details can be found in \cite[Section~2.1.3]{CP2019}.

\subsection{Graphs}
By a \df{graph} we mean a finite simple undirected graph, i.e., a pair $X=(\Omega,E)$
of a finite set $\Omega$ of vertices and an irreflexive symmetric relation $E\subseteq \Omega\times \Omega$,
which represents the edge set of $X$.
The elements of $E=:E(X)$, which are ordered pairs of vertices, are
called \df{arcs}, and $E$ is the \df{arc set} of the graph $X$.
Two vertices $\alpha,\beta\in \Omega$ are said to be \df{adjacent} (in $X$)
whenever $(\alpha,\beta)\in E$; 
we also say that $\beta$ is an $X$-\df{neighbor} of $\alpha$.
A vertex is said to be \df{pendant}, if it has a unique $X$-neighbor.
The graph $X$ is \df{regular} if the number of $X$-neighbors of $\alpha$ is the same for all vertices $\alpha\in \Omega$.
The \df{distance} between any two vertices of $X$
is defined as usual to be the length of a shortest path in~$X$ from one to the other.
For $\Delta\subseteq \Omega$, let $X\setminus \Delta$ denote the subgraph of $X$ induced by $\Omega\setminus \Delta$.

Two vertices $\alpha$ and $\beta$ of $X$
are called \df{twins} (in $X$) if, for any vertex $\gamma\in\Omega\setminus\{\alpha,\beta\}$,
the set $\gamma E$ contains either both $\alpha$ and $\beta$ or none of them.
The relation $e$ ``to be twins in $X$'' is an equivalence relation on $\Omega$.
An equivalence relation contained in $e$ is called a \df{twin equivalence} of~$X$.

Let $e$ be a twin equivalence of~$X$. Denote by $X/e$ the graph with vertex set $\Omega/e$, in which
two distinct vertices $\alpha e$, $\beta e$ are adjacent whenever every two vertices, one in $\alpha e$ and the other in $\beta e$,
are adjacent in $X$. We say that $X/e$ is the \df{quotient graph} of the graph $X$.


\begin{lemma}\label{lm:quotgraph}
 Let $e$ be a twin equivalence of a graph~$X$. Then the quotient graph $X/e$ is isomorphic to an induced subgraph of $X$.
\end{lemma}
\begin{proof}
  By the definition of $X/e$,
  $\alpha e$ and $\beta e$ are adjacent in $X/e$
  if and only if $\alpha'$ and $\beta'$ are adjacent in $X$ for all $\alpha'\in \alpha e$, $\beta'\in \beta e$.
  Thus, $X/e$ can be seen as a graph obtained from $X$ by removing from each equivalence class of $e$
  all but one (arbitrarily chosen) vertex.
\end{proof}

A graph $X$ is called \df{distance-hereditary} if the distance between any two vertices
in any connected induced subgraph of $X$ is the same as it is in $X$. The lemma below immediately follows from the definition.


\begin{lemma}\label{lm:DHinduced}
	An induced subgraph of a distance-hereditary graph is distance-hereditary.
\end{lemma}

Let us recall the three one-vertex extensions by means of which
all finite connected distance-hereditary graphs can be constructed.
Let $X$ be a graph and let $\alpha$ be any vertex of $X$.
Extend $X$ to a graph $X'$ by adding a new vertex $\alpha'$ to $X$ with new edges from $\alpha'$ to either
\begin{itemize}
	\item only $\alpha$, 
	\item $\alpha$ and all its $X$-neighbors,
	\item just all $X$-neighbors of $\alpha$.
\end{itemize}
In the first case the new  vertex $\alpha'$ has degree 1 in $X'$ and we say that
$X'$ is obtained from $X$ by \df{attaching a pendant vertex}, which is $\alpha'$.
In the remaining two cases the vertices $\alpha, \alpha'$ are twins in $X'$,
and we say that $X'$ is obtained from $X$ by \df{splitting a vertex}.



\begin{theorem}\cite[Theorem~1]{BM1986}\label{theo:DHinduction}
	A finite connected graph is distance-hereditary if and only if it is obtained
	from the one-vertex graph by a sequence of one-vertex extensions: attaching pendant vertices and splitting vertices.
\end{theorem}

\begin{corollary}\label{coro:DHinduction}
	A distance-hereditary graph with at least two vertices has
	either a pendant vertex or 
    two distinct twins.
\end{corollary}

\subsection{Coherent closure}\label{ss:closure}
There is a natural partial order\, $\le$\, on the set of all coherent configurations on the same set~$\Omega$.
Namely, given two coherent configurations $\cX=(\Omega,\sS)$ and $\cX'=(\Omega,\sS')$, we set
$$
\cX\le\cX'\ \Leftrightarrow\ \sS^\cup\subseteq (\sS')^\cup.
$$
The \df{coherent closure} $\wL(T)$ of a set $T$ of relations on $\Omega$, is defined to be the smallest coherent configuration
on $\Omega$, which contains $T$ as a set of relations. The operator $\wL$ is monotone in the sense that if $S^\cup\subseteq T^\cup$,
then $\wL(S)\le\wL(T)$ and, moreover, $\wL(\rho_e(S))\le\wL(\rho_e(T))$ if $e$ is an equivalence relation on~$\Omega$.

Let $X$ be a graph. The coherent configuration $\wL(X)$ of $X$ is defined to be the coherent closure of the set $T=\{E(X)\}$.
For a partition $\pi$ of the vertex set of $X$, we denote by $\wL(X)_\pi$ the coherent closure of the set $T$ consisting
of $E(X)$ and all $1_\Delta$, $\Delta\in\pi$.

\begin{lemma}\label{lm:aut}
In the above notation, denote by $\aut(X)_\pi$ the subgroup of $\aut(X)$ leaving each class of~$\pi$ fixed.
Then	
$\aut(\wL(X)_\pi)=\aut(X)_\pi$.
\end{lemma}
\begin{proof}
Follows from~\cite[Theorem~2.6.4]{CP2019}.
\end{proof}

In the present paper, we avoid (vertex) colored graphs, instead we prefer to speak on a graph~$X$ equipped with partition $\pi$  of  the vertex set (of $X$). We say that $\pi$ is  \df{correct} if
\[
\pi=\sF(\wL(X)_\pi),
\]
In what follows, we also say that $\pi$ a  \df{correct partition} of the graph $X$.
	
The exact definition of the WL-dimension of a graph requires a discussion about the $d$-dimensional Weisfeiler-Leman algorithm,
which is beyond the scope of the present paper; we refer the interested reader to the monograph~\cite{Grohe2017}.
In the theorem below, we cite a characterization of regular graphs of the WL-dimension~$1$ and a characterization of graphs
of the WL-dimension at most~$2$. The proofs can be found in \cite[Lemma~3.1(a)]{AKRV2017}  and \cite[Theorem~2.1]{FKV2019}, respectively.

%

\begin{theorem}\label{theo:dim1or2}
Let $X$ be a graph and $d$ the WL-dimension of~$X$.
\begin{itemize}
  \item[$(1)$] If $X$ is regular, then $d=1$ if and only if  $X$ or its complement
  is isomorphic to a complete graph, 
  a cocktail party graph\footnote{A {\it matching graph} in the terminology of \cite{AKRV2017}.}, or the 5-cycle.
  \item[$(2)$] $d\le 2$ if and only if the coherent configuration $\wL(X)$ is separable.
\end{itemize}
\end{theorem}


\section{Twins in coherent configurations}\label{sect:twins}
Let $\cX=(\Omega,\sS)$ be a rainbow.
Two points $\alpha$ and $\beta$ are called $\kk{X}$-\df{twins} if
\[
\boldsymbol{r}(\gamma,\alpha)=\boldsymbol{r}(\gamma,\beta)\quad\text{for all}\ \ \,\gamma\in\Omega\setminus\{\alpha,\beta\}.
\]
Obviously, any two $\cX$-twins belong to the same fiber of $\cX$.
Furthermore, the relation ``being $\kk{X}$-twins'' is an equivalence relation on $\Omega$.
We denote it by~$e_{\kk{X}}$.

\begin{lemma}\label{lm:twin_is_parabolic}
Let $\kk{X}$ be a coherent configuration. Then $e:=e_{\cX}$ is a parabolic of $\kk{X}$.
Moreover, for  all irreflexive relations $r,s\in\sS(\cX)$, we have $r=s$ if and only if $\rho_e(r)=\rho_e(s)$.
\end{lemma}
\begin{proof}
To prove the first statement, 
  for $s\in \sS:=\rel{\kk{X}}$, 
  put
  \[
  e(s):=\{(\alpha,\beta)\in \Omega^2\!: \alpha s=\beta s\}.
  \]
  It follows from \cite[Exercise~2.7.8(1)]{CP2019} that $e(s)$ is a parabolic
  of $\cX_{\Omega'}$, where $\Omega':=\Omega_{-}(e(s))$.
  By \cite[Proposition~2.1.18]{CP2019}, for every fiber $\Delta\in \fiber{\cX}$, the relation
  \begin{equation}\label{eq-eD}
    e(\Delta):=\bigcap_{s\in \sS:~\Omega_{-}(s)=\Delta} e(s)
  \end{equation}
  is a parabolic of $\cX_{\Delta}$.
  The classes of $e(\Delta)$ obviously coincide with those of $e_\Delta$.
  Therefore $e$ equals the union of $e(\Delta)$ for all $\Delta\in \fiber{\kk{X}}$.
  Thus, the equivalence relation $e$ is a relation of $\cX$ and hence a parabolic.

  To prove the second statement, it suffices to verify that if relations $r,s\in\sS$ are irreflexive and $\rho_e(r)=\rho_e(s)$,
  then $r=s$. Suppose on the contrary that $r\ne s$. Without loss of generality, we may assume that there exists a pair $(\alpha,\beta)$ belonging to $r\setminus s$. Since the pair $(\alpha e,\beta e)$ belongs both $\rho_e(r)$ and $\rho_e(s)$,
  one can find points $\alpha'\in\alpha e$ and $\beta'\in\beta e$ such that $\boldsymbol{r}(\alpha',\beta')=s$. On the other hand, since $e=e_\cX$, this implies that
$$
r=\boldsymbol{r}(\alpha,\beta')=\boldsymbol{r}(\alpha',\beta)
$$
and hence
$$
r^*=\boldsymbol{r}(\beta,\alpha')=\boldsymbol{r}(\beta',\alpha')=s^*,
$$
whence $r=s$, a contradiction.
\end{proof}

In view of Lemma \ref{lm:twin_is_parabolic}, the equivalence relation $e_{\kk{X}}$  is called the \df{twin parabolic}
of the coherent configuration~$\kk{X}$. A characterization of the twin parabolic among the other parabolics of a coherent
configuration is given in the following statement.

\begin{lemma}\label{lm:e=eX_iff}
Let $e$ be a parabolic of a coherent configuration $\cX$.
Then $e=e_\cX$ if and only if $e$ is a maximal parabolic of $\cX$ satisfying the following two conditions:
\begin{itemize}
  \item[$(1)$] $e\cdot s=s$ for all $s\in\sS(\cX)$ such that $s\cap e=\varnothing$;
  \item[$(2)$] $e_\Delta=1_\Delta$ or $e_{\Delta}\setminus 1_\Delta\in\sS(\cX)$ for all $\Delta\in\sF(\cX)$.
\end{itemize}
\end{lemma}
\begin{proof}
To prove the ``only if'' part, assume that $e=e_\cX$. Then Condition $(1)$ is obvious.
To prove that Condition~$(2)$ holds, suppose on the contrary that $e_{\Delta}\setminus 1_{\Delta}\not\in\sS(\cX)$.
Then there are two distinct irreflexive basis relations $r$ and $s$ such that $r\cup s\subseteq e$.
It follows that
$$
\rho_e(r)=\rho_e(e_\Delta)=\rho_e(s),
$$
which, by Lemma~\ref{lm:twin_is_parabolic}, implies that $r=s$, a contradiction.
Finally, to prove the maximality of~$e$, let $e\subsetneq e'$ for some parabolic $e'$ of~$\cX$.
Then there exist $\alpha$ and $\beta$ that are not $\cX$-twins but $(\alpha,\beta)\in e'$.
It follows that the relations $r=\boldsymbol{r}(\alpha,\gamma)$ and $s=\boldsymbol{r}(\beta,\gamma)$
are distinct for some point~$\gamma$. But then $e\cdot s\supset r\cup s\ne s$ and so Condition~$(1)$ is violated for $e'$.

To prove the ``if'' part, assume that $e$ is a maximal parabolic of $\cX$ satisfying Conditions~$(1)$ and~$(2)$.
Then obviously every two points $\alpha$ and $\beta$ such that $(\alpha,\beta)\in e$ are $\cX$-twins.
Consequently, $e\subseteq e_\cX$. Now the required statement follows from the maximality of~$e$.
\end{proof}

Every algebraic isomorphism preserves the inclusion between the relations, the dot product and parabolics.
Thus, the  following corollary is an immediate consequence of Lemma~\ref{lm:e=eX_iff}.

\begin{corollary}\label{coro:twinparabolicISO}
Let $\varphi \in \mathrm{Iso}_{\mathrm{alg}}(\kk{X},\kk{X}')$. Then $\varphi(e_{\cX})=e_{\kk{X}'}$.
\end{corollary}

We complete the section by a statement showing a relationship between twins in graphs and in coherent configurations.

\begin{lemma}\label{lm:twins}
Let $X$ be a graph, $\pi$ a correct partition of $X$,
and $\kk{X}=\wl{X}_{\pi}$. Then
\begin{itemize}
  \item[$(1)$] every two $\kk{X}$-twins are twins in $X$;
  \item[$(2)$] every two twins in $X$ belonging to the same fiber of $\kk{X}$ are $\kk{X}$-twins.
\end{itemize}
\end{lemma}
\begin{proof}
  Part (1) follows by $E(X)\in \rel{\kk{X}}^{\cup}$. Let $\Omega$ be the vertex set of $X$, and let $\alpha,\beta$ be distinct twins in $X$. Then   the transposition $(\alpha,\beta)\in \mathrm{Sym}(\Omega)$ is an automorphism of $X$
  fixing pointwise the set $\Omega \setminus \{\alpha, \beta\}$ and permuting $\alpha$ and $\beta$.
  If $\alpha,\beta$ are contained in the same fiber of $\kk{X}$, then this transposition
  preserves the partition $\pi=\sF(\cX)$ of $\Omega$, hence it is an automorphism of $\kk{X}$ by Lemma~\ref{lm:aut}.
  This immediately proves Part (2).
\end{proof}

\section{Two operations}\label{sect:main}
In this section we consider two operations on graphs:  \df{removing a matching} and
\df{reducing twins}, which consist in removing subsets of vertices satisfying certain conditions.
In both cases, the resulting graph is an induced subgraph of the original one.
Our goal is to show how these operations affect the coherent configurations of graphs.
Throughout this section $X=(\Omega,E)$ is a graph.

\subsection{Matchings}
Let $\cX=(\Omega,\sS)$ be a coherent configuration.
A basis relation  $m\in\sS$ is called a  \df{matching} if $m$ is irreflexive and $n_{m^{}}=n_{m^*}=1$.
A unique point  in $\alpha m$, $\alpha\in\Omega_{-}(m)$, is simply denoted by $\alpha m$.
Note that a matching $m$ defines a bijection from $\Omega_{-}(m)$ to $\Omega_{+}(m)$.
Furthermore, one can see that if a relation $m\cdot s$ (similarly, $s\cdot m$) with $s\in\sS$ is nonempty,
then it is a basis one.

Suppose further that $E\in \sS^{\cup}$,
and let $m$ be a matching of $\cX$ such that $\Omega_{-}(m)\ne \Omega_{+}(m)$.
If, for any $\delta\in \Omega_{-}(m)$, the vertex $\delta m$ is a unique $X$-neighbor
of $\delta$ (i.e., the vertices of $\Delta$ are all pendant),
then $m$ is called a \df{pendant matching} of $\cX$.
If, for any $\delta\in \Omega_{-}(m)$, the vertex $\delta m$ is a twin of~$\delta$ in $X$,
then $m$ is called a \df{twin matching} of $\cX$.
It is more precise to speak about pendant or twin matchings of $\cX$ {\it with respect to the graph} $X$;
in what follows we omit $X$ if it is clear from the context.

\begin{proposition}\label{prop:matching} {\rm(Removing matching)}
Let $X$ be a graph, $\pi$ a correct partition of $X$,
$m$ a matching of $\wl{X}_{\pi}$, and $\Delta:=\Omega_{-}(m)$.
If $m$ is pendant or twin, then
\begin{equation}\label{eq-matchingprop}
  \wl{X}_{\pi}\setminus \Delta=\wl{X\setminus \Delta}_{\pi\setminus \Delta},  
\end{equation}
and $\pi\setminus\Delta$ is a correct partition of $X\setminus \Delta$.
\end{proposition}
\begin{proof}
Put
\[
\kk{X}:=(\Omega,\sS)=\wl{X}_{\pi},\quad Y=X\setminus\Delta,\quad
\kk{Y}:=(\Omega\setminus \Delta,\sS')=\wl{X\setminus \Delta}_{\pi\setminus \Delta},
\]
so that Eq. \eqref{eq-matchingprop} can be rewritten as $\cX\setminus\Delta=\cY$,
and we aim to prove this equality.

Since $\pi$ is a correct partition, each of the relations $1_\Gamma$, $\Gamma\in\pi\setminus\Delta$,
belongs to $\sS(\cX\setminus\Delta)^\cup$. As $E(Y)=E(X\setminus \Delta)\in \sS(\cX\setminus\Delta)^\cup$
and $\kk{Y}$ is the minimal coherent configuration containing $E(Y)$ and all $1_\Gamma$, $\Gamma\in\pi\setminus\Delta$,
among its relations, we have
  \begin{equation}\label{eq-XDeltaatleastY}
    \kk{X}\setminus \Delta\geq \kk{Y}.
  \end{equation}


  Let $\sT$ be the set of nonempty binary relations on $\Omega$ belonging to the set
  \begin{equation*}\label{eq-Crel}
    \sS'\cup m\cdot \sS'\cup \sS'\cdot m^*\cup m\cdot \sS'\cdot m^*.
  \end{equation*}
  One can see that $\sT$ is a partition of $\Omega\times \Omega$ that satisfies Eq. \eqref{eq-rainbow};
  this allows us to define an auxiliary rainbow $\cX'=(\Omega,\sT)$.
  Moreover, observe that 
  \begin{eqnarray}
    \sT_{\Delta,\Delta} &=& m\cdot \sS'\cdot m^*, \label{eq-Tparts1}\\
    \sT_{\Delta,\Omega\setminus\Delta} &=& \sT_{\Omega\setminus\Delta,\Delta}^* = m\cdot \sS', \label{eq-Tparts2}\\
    \sT_{\Omega\setminus\Delta,\Omega\setminus\Delta} &=& \sS'. \label{eq-Tparts3}
  \end{eqnarray}

\begin{claim}\label{cl:prop11}
Suppose that $\cX'$ is a coherent configuration and $E\in \sT^{\cup}$.
Then the conclusion of the proposition holds.
\end{claim}
\begin{proof}
  By $E\in \sT^{\cup}$, we see that $\cX'\geq \wl{X}$.
  Furthermore, the partition $\sF(\cX')$ of $\Omega$ is a refinement of $\pi=\sF(\cX)$,
  since $\sF(\cX')=\sF(\cY)\cup \{\Delta\}$ by the construction of $\cX'$ and
  $\sF(\cY)$ is a refinement of $\pi\setminus\Delta$ by the definition of $\cY$.
  This implies that $\cX'\geq \wl{X}_{\pi}=\cX$. Hence, we obtain
  \begin{equation}\label{eq-DgeqCA}
    \cX\setminus \Delta \leq \cX'\setminus \Delta=\cY,
  \end{equation}
  which, together with Eq. \eqref{eq-XDeltaatleastY}, yields $\cX\setminus \Delta=\cY$, as required.
  Moreover, then $\fiber{\cY}=\pi\setminus \Delta$ holds, i.e.,
  $\pi\setminus\ \Delta$ is a correct partition of the graph $X\setminus \Delta$.
\end{proof}

We need the following two claims.

\begin{claim}\label{cl:prop12}
$\cX'$ is a coherent configuration.
\end{claim}
\begin{proof}
We need to verify that the number
\[
a=|\alpha r\cap \beta s^{*}|
\]
does not depend on the choice of $(\alpha,\beta)\in t$ for all $r,s,t\in \sT$.
Let $c^{t'}_{r',s'}$ with $r',s',t'\in \sS'$ stand for the intersection numbers of $\cY$.

Eqs. \eqref{eq-Tparts1}--\eqref{eq-Tparts3} imply that $a=0$ for all $(\alpha,\beta)\in t$
or, up to replacing any of $r,s,t\in \sT$ by $r^*,s^*,t^*$, respectively,
one of the following cases holds:
\begin{itemize}
  \item $r,s,t\in \sS'$: here clearly $a=c^t_{r,s}$ holds;
  \item $r\in \sS'\cdot m^*$, $s\in m\cdot \sS'$ and $t\in \sS'$: here
  $r=r'\cdot m^*$, $s=m\cdot s'$ for some $r',s'\in \sS'$, and hence
  \[
  a=|\alpha r\cap \beta s^{*}|=|\alpha (r'\cdot m^*)\cap \beta (s'^*\cdot m^*)|=|\alpha r'\cap \beta s'^*|=c^t_{r',s'};
  \]
  \item $r\in m\cdot \sS'$, $s\in \sS'$ and $t\in m\cdot \sS'$: here
  $r=m\cdot r'$, $t=m\cdot t'$ for some $r',t'\in \sS'$, and hence
  \[
  a=|\alpha (m\cdot r')\cap \beta s^*|=|(\alpha m) r'\cap \beta s^*|=c^{t'}_{r',s};
  \]
  \item $r\in m\cdot \sS'\cdot m^*$, $s\in m\cdot \sS'$ and $t\in m\cdot \sS'$: here
  $r=m\cdot r'\cdot m^*$, $s=m\cdot s'$, $t=m\cdot t'$ for some $r',s',t'\in \sS'$, and hence
  \[
  a=|\alpha (m\cdot r'\cdot m^*)\cap \beta (s'^*\cdot m^*)|=|(\alpha m) r'\cap \beta s'^*|=c^{t'}_{r',s'};
  \]
  \item $r\in m\cdot \sS'$, $s\in \sS'\cdot m^*$ and $t\in m\cdot \sS'\cdot m^*$: here
  $r=m\cdot r'$, $s=s'\cdot m^*$, $t=m\cdot t'\cdot m^*$ for some $r',s',t'\in \sS'$, and hence
  \[
  a=|\alpha (m\cdot r')\cap \beta (m\cdot s'^*)|=|(\alpha m) r'\cap (\beta m) s'^*|=c^{t'}_{r',s'};
  \]
  \item $r,s,t\in m\cdot \sS'\cdot m^*$: here
  $r=m\cdot r'\cdot m^*$, $s=m\cdot s'\cdot m^*$, $t=m\cdot t'\cdot m^*$ for some $r',s',t'\in \sS'$,
  and, as above, $a=c^{t'}_{r',s'}$ holds.
\end{itemize}

Thus, in any case the number $a$ does not depend on the choice of $(\alpha,\beta)\in t$,
and we are done.
\end{proof}

\begin{claim}\label{cl:prop13}
$E\in \sT^{\cup}$ holds.
\end{claim}
\begin{proof}
  We first note that $E_{\Omega\setminus\Delta}=E(Y)\in \sT^{\cup}$ holds,
  since $E(Y)\in (\sS')^{\cup}$ by the definition of $\cY$ and $\sS'\subseteq \sT$ by the definition of $\cX'$.
  If $m$ is a pendant matching, then obviously $E=m\cup m^*\cup E(Y)$, so we are done by $m\in \sT$.

  Suppose that $m$ is a twin matching.
  Choose an arbitrary $s\in \sT$ with $\Omega_{-}(s)=\Omega_{+}(m)$.
  Observe that, for any $\delta\in \Delta$, $\gamma:=\delta m$, and $\beta\in \gamma s$,
  we have that $\boldsymbol{r}_{\cX'}(\delta,\beta)=m\cdot s$.
  As $\delta$ and $\gamma$ are twins in $X$, i.e., $(\gamma,\beta)\in E \Leftrightarrow (\delta,\beta)\in E$,
  we see that
  \begin{equation}\label{eq1}
    s\subseteq E\ \Leftrightarrow\ m\cdot s\subseteq E\qquad\text{and}\qquad s\cap E=\varnothing\ \Leftrightarrow\  (m\cdot s)\cap E=\varnothing.
  \end{equation}

  If $\Omega_{+}(s)\subseteq \Omega\setminus \Delta$, then $s\in \sS'$ by Eq. \eqref{eq-Tparts3}.
  Since $s$ is a basis relation of $\cY$ and $E(Y)\subseteq E$, it follows that
  $(s\cap E\ne \varnothing \Rightarrow s\subseteq E)$ holds.
  By Eq. \eqref{eq1}, we obtain
  \begin{equation}\label{eq2}
s'\in \sS',\ (m\cdot s')\cap E\ne \varnothing\qquad \Rightarrow\qquad m\cdot s'\subseteq E.
  \end{equation}

  If $\Omega_{+}(s)=\Delta$, then $s=s'\cdot m^*$ for $s'=\boldsymbol{r}_{\cX'}(\gamma,\beta m)$, $s'\in \sS'$ by Eq. \eqref{eq-Tparts1}.
  As $\beta$ and $\beta m$ are twins in $X$, we see that
  \begin{equation}\label{eq3}
s\subseteq E\ \Leftrightarrow\ s'\subseteq E\qquad\text{and}\qquad s\cap E=\varnothing\ \Leftrightarrow\ s'\cap E=\varnothing.
  \end{equation}
  As above, 
  it follows that
  $(s'\cap E\ne \varnothing \Rightarrow s'\subseteq E)$ holds.
  By Eqs. \eqref{eq1}, \eqref{eq3}, 
  this implies that
  \begin{equation}\label{eq4}
s''\in \sS',\ (m\cdot s''\cdot m^*)\cap E\ne \varnothing\qquad \Rightarrow\qquad m\cdot s''\cdot m^*\subseteq E.
  \end{equation}

  By Eqs. \eqref{eq2} 
  and \eqref{eq4}, we conclude that
  \[
     E_{\Delta,\Omega}\in (m\cdot \sS'\cup m\cdot \sS'\cdot m^*)^{\cup}\subseteq \sT^{\cup},
  \]
  so that $E=E_{\Delta,\Omega}\cup E_{\Delta,\Omega}^*\cup E_{\Omega\setminus\Delta}\in \sT^{\cup}$ holds,
  as required.
%
\end{proof}

  Claims \ref{cl:prop12} and \ref{cl:prop13} show that $\cX'$ satisfies the assumption of Claim \ref{cl:prop11},
  whence the proposition follows.
\end{proof}

\begin{proposition}\label{prop:matching_separability}
  In the notation of Proposition \ref{prop:matching},
  $\wl{X}_{\pi}\setminus \Delta$ is separable
  if and only if
  $\wl{X}_{\pi}$ is separable.
\end{proposition}
\begin{proof}
  The result follows from \cite[Lemma 3.3(1)]{FKV2019}.
\end{proof}

\subsection{Twins}
Let $\pi$ be a correct partition of the graph $X$ and $e$ the twin parabolic of $\wl{X}_{\pi}$.
Recall (see Section~\ref{ss:parabolics}) that $\pi/e:=\rho_e(\pi)$, where $\rho_e$ is the mapping
defined by Eq. \eqref{eq-ro}. Since $\rho_e$ preserves the fibers,
$\pi/e$ is the set of fibers of $\qo{\wl{X}_{\pi}}$.
The next proposition is an analogue of Proposition~\ref{prop:matching}
in regard to twins in a graph.

\begin{proposition}\label{prop:quot} {\rm (Reducing twins)}
Let $X$ be a graph, $\pi$ a correct partition of $X$,
and $e$ the twin parabolic of $\wl{X}_{\pi}$.
Then $e$ is a twin equivalence of $X$,
\begin{equation}\label{eq-twinprop}
  \qo{\wl{X}_{\pi}}=\wl{\qo{X}}_{\pi/e},
\end{equation}
and $\pi/e$ is a correct partition of $X/e$.
\end{proposition}
\begin{proof}
The statement about $e$ being a twin equivalence of the graph $X$
follows from Lemma~\ref{lm:twins}. Next, we put
\[
\kk{X}:=(\Omega,\sS)=\wl{X}_{\pi},\quad Y=\qo{X},\quad \kk{Y}:=(\Omega/e,\overline{\sT})=\wl{Y}_{\pi/e},
\]
so that Eq. \eqref{eq-twinprop} can be rewritten as $\cX/e=\cY$,
and we aim to prove this equality.

By the definition of $\cX/e$ (see Section \ref{ss:parabolics}), it follows that
$E(Y)=\rho_e(E)\in \sS(\cX/e)^{\cup}$ and $1_{\Delta}\in \sS(\cX/e)^{\cup}$ for every $\Delta\in\pi/e$.
Since $\cY$ is the minimal coherent configuration containing $E(Y)$ and all $1_\Delta$, $\Delta\in\pi/e$,
among its relations, we have
  \begin{equation}\label{eq-XeatleastY}
   \cX/e\geq \cY.
  \end{equation}

  Put $\rho:=\rho_e$ and define a set $\sT$ of binary relations on $\Omega$ as the union
\[
\{\rho^{-1}(\ov t)\!:\ov t\in\overline{\sT}\ \,\text{is irreflexive}\}\quad\cup\quad
\{e_\Delta\setminus 1_\Delta\!:\ \Delta\in\sF(\cX)\ \,\text{with}\ \, e_{\Delta}\ne 1_{\Delta}\}\quad\cup\quad
\{1_\Delta\!:\ \Delta\in\sF(\cX)\}.
\]
  One can see that $\sT$ is a partition of $\Omega\times \Omega$ that satisfies Eq. \eqref{eq-rainbow};
  this allows us to define an auxiliary rainbow $\cX'=(\Omega,\sT)$.
  Moreover, observe that $e\in\sT^{\cup}$ is a parabolic of $\cX'$.

%
%

\begin{claim}\label{cl:prop21}
Suppose that $\cX'$ is a coherent configuration and $E\in \sT^{\cup}$.
Then the conclusion of the proposition holds.
\end{claim}
\begin{proof}
  By $E\in \sT^{\cup}$, we see that $\cX'\geq \wl{X}$.
  Obviously, $\sF(\cX')$ is a refinement of $\pi=\sF(\cX)$ and hence $\cX'\geq \wl{X}_{\pi}=\cX$.
  Furthermore,  since $\rho(t)\in \overline{\sT}$ for all $t\in \sT$, $\cX'/e=\cY$ holds.
  Therefore, by the monotonicity of $\leq$ (see Section \ref{ss:closure}), we obtain
  \begin{equation}\label{eq-DgeqCA}
    \cX/e\leq \cX'/e=\cY,
  \end{equation}
  which, together with Eq. \eqref{eq-XeatleastY}, yields $\cX/e=\cY$, as required.
  Moreover, then $\fiber{\cY}=\pi/e$ holds, i.e.,
  $\pi/e$ is a correct partition of the graph $X/e$.
\end{proof}

We need the following two claims.

\begin{claim}\label{cl:prop22}
$\cX'$ is a coherent configuration.
\end{claim}
\begin{proof}
It suffices to verify that the number
\[
a=|\alpha r\cap \beta s^{*}|
\]
does not depend on the choice of $(\alpha,\beta)\in t$ for all $r,s,t\in \sT$.  To this end,
we set $\ov\delta=\rho(\delta)$ for every $\delta\in\Omega$, and note
that $\ov r=\rho(r)$, $\ov s=\rho(s)$, and $\ov t=\rho(t)$ are basis relations of~$\cY$.
Therefore,
\[
|\ov\alpha \ov r\cap \ov\beta \ov s^{*}|=c^{\ov{t}}_{\ov{r}\ov{s}}.
\]
and this number does not depend on $(\alpha,\beta)\in t$. Furthermore, if $a=0$ for some $(\alpha,\beta)\in t$, then
$c^{\ov{t}}_{\ov{r}\ov{s}}=0$, because $e\subseteq e_{\cX'}$ (see the definition of $\sT$).
But then obviously $a=0$ for all $(\alpha,\beta)\in t$.
Thus, without loss of generality, we may assume that $a\ne 0$.

Let $\gamma\in \alpha r\cap \beta s^{*}$. It is easily seen that  the points $\ov\gamma \in \ov\alpha \ov{r}\cap \ov\beta \ov{s}^{*}$
belong to the $\rho$-image of the fiber of~$\cX$, containing~$\gamma$. By Eq.~\eqref{220320a}, this implies that
the number $k=|\gamma e|$ does not depend on $\gamma\in \alpha r\cap \beta s^{*}$.
It follows that
$$
a=\css
kc^{\overline{t}}_{\overline{r},\overline{s}}  &\text{ if $r\not\subseteq e\ \wedge\ s\not\subseteq e$},\\
k-1  &\text{ if $(r\not\subseteq e\ \wedge\ s\subseteq e')\ \vee \ (r\subseteq e'\ \wedge\ s\not\subseteq e)$},\\ 
k-2  &\text{ if $r\subseteq e'\ \vee \ s\subseteq e'\ \vee \ t\subseteq e'$},\\ 
1  &\text{ otherwise},\\
\ecss
$$
where $e'=e\setminus 1_\Omega$; here, we made use the fact that if $x\in \sT$ is contained in $e$, then $x\subseteq e'$ or $x\subseteq 1_\Omega$. Thus, the number $a$ does not depend on the choice of $(\alpha,\beta)\in t$, as required.
\end{proof}

\begin{claim}\label{cl:prop23}
$E\in \sT^{\cup}$ holds.
\end{claim}
\begin{proof}
It suffices to verify that $E$ contains each irreflexive $t\in \sT$ such that $t\cap E\ne\varnothing$.
Note that the latter condition implies that $\rho(t)\cap \rho(E)\ne\varnothing$. Moreover,  $\rho(t)$
is a basis relation of $\cY$ and $\rho(E)$ is a relation of $\cY$. Thus, $\rho(t)\subseteq \rho(E)$
and hence
\[
\rho^{-1}(\rho(t))\subseteq \rho^{-1}(\rho(E)).
\]
On the other hand, the relation $\rho^{-1}(\rho(t))$ is equal to $t$ if $\ov t$ is irreflexive,
or $e_\Delta$ for some $\Delta\in\sF(\cX)$ otherwise. In any case, $t\subseteq \rho^{-1}(\rho(t))$. Furthermore,
\[
\rho^{-1}(\rho(E))\subseteq E\cup 1_\Omega,
\]
because $e$ is a twin equivalence of $X$ (see above). Thus,
\[
t\subseteq \rho^{-1}(\rho(t))\subseteq \rho^{-1}(\rho(E))\subseteq E\cup 1_\Omega.
\]
Since $t$ is irreflexive, this implies that $t\subseteq E$, as required.
\end{proof}

  Claims \ref{cl:prop22}, \ref{cl:prop23} show that $\cX'$ satisfies the assumption of Claim \ref{cl:prop21},
  whence the proposition follows.
\end{proof}

The following proposition, which holds for any coherent configuration, together
with Proposition~\ref{prop:quot} show that $\wl{X}_{\pi}$ is separable if $\wl{\qo{X}}_{\pi/e}$ is separable,
and this fact will be used in the proof of Theorem \ref{theo:main} in Section \ref{sect:proof}.

\begin{proposition}\label{prop:quot_separability}
  A coherent configuration $\kk{X}$ is separable if $\kk{X}/\eqv$ is separable.
\end{proposition}
\begin{proof}
Let $\kk{X}=(\Omega,\mathsf{S})$ and $e:=\eqv$.   Assume that the coherent configuration $\qo{\kk{X}}$ is separable.
We need to verify that given a coherent configuration $\kk{X}'=(\Omega',\mathsf{S}')$, any algebraic isomorphism $\varphi:s\mapsto s'$
from $\kk{X}$ to $\kk{X}'$ is induced by a bijection. We note that by Corollary~\ref{coro:twinparabolicISO},
$e':=\varphi(e)$ is the twin parabolic of $\kk{X}'$.

Given $\Delta\in \fiber{\kk{X}}$, we choose a full system $\overline{\Delta}$ of distinct representatives of
the classes of the equivalence relation $e_\Delta$,  and put $\overline{\Omega}$ to be the union of
all $\overline{\Delta}$, $\Delta\in \fiber{\kk{X}}$. Then the pair $\overline{\kk{X}}=(\overline{\Omega},\overline{\mathsf{S}})$
with  $\overline{\mathsf{S}}=\mathsf{S}_{\overline{\Omega}}$ is obviously a rainbow.
Since the parabolic $e$ is twin, the natural bijection
$$
\ov h\!:\ov\Omega\to \Omega/e, \quad \alpha\mapsto \alpha e
$$
is a rainbow isomorphism from $\overline{\kk{X}}$ to $\cX/e$ (see the second part of Lemma~\ref{lm:twin_is_parabolic}).
In particular, $\overline{\kk{X}}$ is a coherent configuration.
In a similar way, one can define the sets $\overline{\Delta'}$, $\Delta'\in \fiber{\kk{X}'}$, and $\overline{\Omega'}$,
the rainbow $\overline{\cX'}$, the isomorphism
$$
\overline{h'}\!:\overline{\Omega'}\to \Omega'/e', \quad \alpha'\mapsto \alpha' e',
$$
and check that $\overline{\cX'}$ is a coherent configuration.

The algebraic isomorphism $\varphi$ induces an algebraic isomorphism $\ov\varphi\in\Aiso(\cX/e,\cX'/e')$ (see Subsection~\ref{ss:parabolics}).
By the proposition assumption, the coherent configuration $\cX/e$ is separable. Consequently,~$\ov\varphi$ is induced by a bijection, say $h$.
It follows that the composition mapping $\ov f=\ov h\circ h\circ (\overline{h'})^{-1}$ induces the restriction of~$\varphi$ to $\overline{\mathsf{S}}$, i.e.,
\qtnl{a1}
\boldsymbol{r}(\alpha,\beta)'=\boldsymbol{r}(\alpha^{\ov f},\beta^{\ov f}),\qquad \alpha,\beta\in\ov\Omega.
\eqtn

Let us extend $\overline{f}$ to a bijection $f\!:\Omega\to\Omega'$.   To do so given $\alpha\in\ov\Omega$,
we choose an arbitrary bijection $f_\alpha\!:\alpha e\to \alpha'e$ that takes $\alpha$ to $\alpha'=\alpha^{\ov f}$ (
such a bijection does exist because $|\alpha e|=|\alpha'e'|$ in view of~\eqref{a1}). Since the union of $\alpha e$, $\alpha\in\ov\Omega$,
equals $\Omega$, the desired bijection $\ov f$ is defined uniquely by the condition $f|_{\alpha e}=f_\alpha$.

To complete the proof it suffices to verify that $\boldsymbol{r}(\alpha,\beta)^{f}=\boldsymbol{r}(\alpha,\beta)'$
for all $\alpha,\beta\in\Omega$.   Denote by $\ov\alpha$ and $\ov\beta$ the unique points of $\ov\Omega$, lying
in $\alpha e$ and $\beta e$, respectively. Then $\alpha$ and $\beta$ are $\cX$-twins of $\ov\alpha$ and $\ov \beta$, respectively.
Moreover, from the definition of $f$, it follows that $\alpha^f$ and $\beta^f$ are $\cX'$-twins of $\ov\alpha^f$ and $\ov \beta^f$,
respectively. Thus, by Eq. \eqref{a1}, we have
\[
\boldsymbol{r}(\alpha,\beta)^{f}=\boldsymbol{r}(\alpha^f,\beta^f)=\boldsymbol{r}(\ov\alpha^f,\ov\beta^f)
=\boldsymbol{r}(\ov\alpha,\ov\beta)'=\boldsymbol{r}(\alpha,\beta)',
\]
as required.
\end{proof}

\section{Proof of Theorem \ref{theo:main}}\label{sect:proof}
To prove Theorem \ref{theo:main}, we need the following two auxiliary lemmas.

\begin{lemma}\label{lm:dim1}
  The WL-dimension of the class of distance-hereditary graphs is greater than $1$.
\end{lemma}
\begin{proof}
  It follows from Theorem~\ref{theo:dim1or2}(1) that a regular graph $X$ has WL-dimension 1
  if and only if~$X$ or its complement is isomorphic to a complete graph, a cocktail party graph,
  or the 5-cycle. Since~$K_{n,n}$, a complete bipartite graph with parts of size $n$, is
  regular, it has WL-dimension greater than $1$ if $n>2$.
  As $K_{n,n}$ is distance-hereditary by Theorem \ref{theo:DHinduction}, the lemma follows.
\end{proof}

\begin{lemma}\label{lm:indstep}
  Let $X$ be a distance-hereditary graph with at least two vertices, $\pi$ a correct partition of~$X$,
  and $\kk{X}=\wl{X}_{\pi}$. Then $\kk{X}$ has a twin matching or a pendant matching,
  or the twin parabolic~$\eqv$ is nontrivial.
\end{lemma}
\begin{proof}
Let $X=(\Omega,E)$ and
suppose first that there are no twins in $X$. Then $X$ has pendant vertices by Corollary \ref{coro:DHinduction}.
No two of them share the same $X$-neighbor, for otherwise  they are twins in~$X$, a contradiction.
It follows that if $\alpha$ is a pendant vertex and $\beta$ is a unique $X$-neighbor of $\alpha$,
then $m=\boldsymbol{r}(\alpha,\beta)$ is a matching in $\cX$. Moreover, $\Omega_{-}(m)\ne \Omega_{+}(m)$,
for otherwise the vertices $\alpha$ and $\beta$ form a connected component of $X$ and hence are twins.
Thus, $m$ is a pendant matching.

Let $X$ have two distinct twins $\alpha$ and $\beta$. If they belong to the same fiber of~$\cX$, then
the twin parabolic $\eqv$ is nontrivial by Lemma~\ref{lm:twins}(2) and we are done.  Thus, we may assume
that no two distinct twins in $X$ belongs to the same fiber of $\kk{X}$. To complete the proof, it suffices
to verify that the relation $m=\boldsymbol{r}(\alpha,\beta)$ is a (twin) matching. Assume on the contrary
that $m$ or $m^*$ has valency at least~$2$. Without loss of generality, we may assume that
there exists $\beta'\in \alpha m$ other than~$\beta$.

Suppose that there exists an $X$-neighbor $\gamma$ of $\beta$, which is not an $X$-neighbor of~$\beta'$.
Then the relation $r=\boldsymbol{r}(\alpha,\gamma)$ is contained in $E$ (because $\boldsymbol{r}(\beta,\gamma)\subseteq E$
and $\alpha$ and $\beta$ are twins  in $X$), whereas $t=\boldsymbol{r}(\beta',\gamma)$ is not.
On the other hand,
\[
\boldsymbol{r}(\alpha,\beta)=m=\boldsymbol{r}(\alpha,\beta')\subseteq \boldsymbol{r}(\alpha,\gamma)\cdot \boldsymbol{r}(\gamma,\beta')=r\cdot t^*.
\]
It follows that there exists $\gamma'\in\Omega$ such that $(\alpha,\gamma')\in r$ and $(\beta,\gamma')\in t^*$.
Since $r\subseteq E$ and $t\cap E=\varnothing$, this contradicts the fact  that $\alpha$ and $\beta$ are twins  in $X$.
Thus, the point $\gamma$ does not exist and hence $\beta'E\subseteq \beta E$. Since $\beta'$ and $\beta$ lie in the same fiber
of $\cX$ and $E$ is a relation of $\cX$, this inclusion  is the equality.
Consequently, $\beta$ and $\beta'$ are distinct twins  in $X$, lying in the same fiber, a contradiction.
\end{proof}

We are now in a position to prove Theorem \ref{theo:main}.
  Let $X$ be a distance-hereditary graph.
  By Lemma~\ref{lm:dim1}, it suffices to prove that the WL-dimension of $X$ is at most $2$,
  or, equivalently, the coherent configuration $\wl{X}$ is separable (see Theorem~\ref{theo:dim1or2}(2)).
  We shall prove a more general statement that, for a correct partition $\pi$ of $X$,
  the coherent configuration $\wl{X}_{\pi}$ is separable,
  which implies the result by $\wl{X}=\wl{X}_{\pi}$, where $\pi=\fiber{\wl{X}}$.

  We use induction on the number $n$ of vertices of $X$.
  Without loss of generality, we may assume that $n\geq 2$
  and the statement holds for all distance-hereditary graphs with
  at most $n-1$ vertices and their correct partitions.
  By Lemma \ref{lm:indstep}, 
  the coherent configuration $\kk{X}=\wl{X}_{\pi}$ has a twin matching or a pendant matching $m$,
  or the twin parabolic $e:=\eqv$ is nontrivial.

  In the former case, let $\Delta$ denote $\Omega_{-}(m)$.
  By Proposition \ref{prop:matching}, $\pi\setminus \Delta$ is a correct partition of the graph $X\setminus \Delta$.
  Since this graph is distance-hereditary by Lemma \ref{lm:DHinduced},
  the coherent configuration $\wl{X\setminus \Delta}_{\pi\setminus \Delta}$ is separable by induction.
  Thus, $\kk{X}$ is separable by Propositions \ref{prop:matching} and \ref{prop:matching_separability}.

  In the latter case, $\pi/e$ is a correct partition of the quotient graph $X/e$ by Proposition \ref{prop:quot}.
  By Lemmas \ref{lm:quotgraph} and \ref{lm:DHinduced}, $X/e$ is distance-hereditary.
  Hence, by induction, the coherent configuration $\wl{X/e}_{\pi/e}$ is separable.
  Thus, $\kk{X}$ is separable by Propositions \ref{prop:quot} and \ref{prop:quot_separability}.

\Acknowledgements


\begin{thebibliography}{1}

\bibitem{AKRV2017}
V.~Arvind, J.~K\"obler, G.~Rattan, and O.~Verbitsky,
\emph{Graph Isomorphism, Color Refinement, and Compactness},
Computational Complexity, {\bf 26}, no.~3, 627--685 (2017).

\bibitem{B2015}
L.~Babai, \emph{Group, graphs, algorithms: the graph isomorphism problem}, in: Proceedings of the {I}nternational {C}ongress of 	{M}athematicians---{R}io de {J}aneiro 2018. {V}ol. {IV}. {I}nvited lectures,
3319--3336. World Sci. Publ., Hackensack, NJ.

\bibitem{Conf2018}
\emph{Conference in Algebraic Graph Theory, Symmetry vs Regularity.
The first 50 years since Weisfeiler-Leman stabilization},
July 1 - July 7, 2018, Pilsen, Czech Republic;
https://www.iti.zcu.cz/wl2018/index.html.

\bibitem{BM1986}
H.-J.~Bandelt and H.~M.~Mulder, \emph{Distance-hereditary graphs},
J. Combin. Theory Ser. B, {\bf 41}, no.~2, 182–-208 (1986).

\bibitem{CFI1992}
J.-Y.~Cai, M.~F\"urer, and N.~Immerman,
\emph{An optimal lower bound on the number of variables for graph identification},
Combinatorica, {\bf 12}, no.~4, 389–-410 (1992).

\bibitem{CP2019}	
G.~Chen and I.~Ponomarenko, \emph{Coherent Configurations},
Central China Normal University Press, Wuhan (2019).

\bibitem{DER}
B.~Das, M.K.~Enduri, and I.V.~Reddy,
\emph{Polynomial-time algorithm for isomorphism of graphs
with clique-width at most three},
Theoretical Computer Science, 819, no. 2, 9–-23 (2020).

\bibitem{EPT2000}
S.~Evdokimov, I.~Ponomarenko, and G.~Tinhofer,
\emph{Forestal algebras and algebraic forests (on a new class of weakly compact graphs)},
Discrete Mathematics, {\bf 225}, no.~1–3, 149-–172 (2000).

\bibitem{FKV2019}
F.~Fuhlbr\"uck, J~ K\"obler, and O.~Verbitsky, \emph{Identiability of graphs with small color classes by the Weisfeiler-Leman algorithm}, in: Proc. $37$th International Symposium on Theoretical Aspects of Computer Science, Dagst\"uhl Publishing, Germany (2020), pp. 43:1-43:18.

\bibitem{Grohe2017}
M.~Grohe,
\emph{Descriptive complexity, canonisation, and definable graph structure theory},
Cambridge University Press, Cambridge (2017).

\bibitem{GN2019}	
M.~Grohe and D.~Neuen,
\emph{Canonisation and Definability for Graphs of Bounded Rank Width},
in:  Proc. 2019 34th Annual ACM/IEEE Symposium on Logic in Computer Science (LICS) (2019), pp.~1--13; doi: 10.1109/LICS.2019.8785682.

\bibitem{GK2019}	
M.~Grohe and S. Kiefer,
\emph{A Linear Upper Bound on the Weisfeiler-Leman Dimension of Graphs of Bounded Genus}, in: Proc. 46th International Colloquium on Automata, Languages, and Programming (ICALP 2019), Dagst\"uhl Publishing, Germany (2019), pp. 117:1--117:15; doi: 10.4230/LIPIcs.ICALP.2019.117

\bibitem{H1977}
E.~Howorka,
\emph{A characterization of distance-hereditary graphs},
Quart. J. Math. Oxford Ser. (2), {\bf 28}, no. 112, 417--420 (1977).

\bibitem{KPS2017}
S.~Kiefer, I.~Ponomarenko, and P.~Schweitzer,
\emph{The Weisfeiler-Leman dimension of planar graphs is at most $3$},
J. ACM, {\bf 66}, No.~6, Article 44 (2019), doi: https://doi.org/10.1145/3333003.

\bibitem{NUU}
S.~Nakano, R.~Uehara, T.~Uno,
\emph{A New Approach to Graph Recognition and Applications to Distance-Hereditary Graphs},
in: Cai J.-Y., Cooper S.B., Zhu H. (eds) Theory and Applications of Models of Computation. TAMC 2007.
Lecture Notes in Computer Science, vol. 4484. Springer, Berlin, Heidelberg

\bibitem{Oum2005}
Sang-il Oum,
\emph{Rank-width and vertex-minors},
J. Combin. Theory Ser. B, {\bf 95}, no. 1, 79–-100 (2005).
	
\end{thebibliography}
\end{document}